\documentclass{amsart}

\usepackage{amssymb, amsmath}

\newcommand{\N}{\mathbb{N}}
\newcommand{\Z}{\mathbb{Z}}

\newcommand{\R}{\mathbb{R}}

\newcommand{\fr}{\mathfrak}
\newcommand{\goth}{\mathfrak}

\newcommand{\dist}{\text{dist}}

\newtheorem{theorem}{Theorem}[section]
\newtheorem{lemma}[theorem]{Lemma}
\newtheorem{prop}[theorem]{Proposition}
\newtheorem{cor}[theorem]{Corollary}
\newtheorem{remark}[theorem]{Remark}

\numberwithin{equation}{section}
\allowdisplaybreaks

\begin{document}

\title{Simple mixing actions with uncountably many prime factors}

\author[A. I. Danilenko]{Alexandre I. Danilenko}
\address{Institute for Low Temperature Physics \& Engineering of the National Academy of Sciences of Ukraine, 47 Lenin Ave., Kharkiv, 61103, UKRAINE}
\email{alexandre.danilenko@gmail.com}

\author[A. V. Solomko]{Anton V. Solomko}
\email{solomko.anton@gmail.com}

\subjclass[2010]{37A05, 37A10}

\begin{abstract}
Via $(C,F)$-construction we produce a 2-fold simple mixing transformation which has uncountably many non-trivial proper factors and all of them are prime.
\end{abstract}

\maketitle

\section{Introduction}
This paper is about prime factors of simple probability preserving actions.
We first recall these and related definitions from the theory of joinings.

Let $T=(T_g)_{g\in \Gamma}$ be an ergodic action of a locally compact second countable group $\Gamma$ on a standard probability space $(X,\mathfrak B,\mu)$.
The main interest for us lies in $\Z$ and $\R$-actions.
A measure $\lambda$ on $X\times X$ is called a \emph{2-fold self-joining} of $T$ if it is $(T_g\times T_g)_{g\in\Gamma}$-invariant and it projects onto $\mu$ on both coordinates.
Denote by $J_2^e(T)$ the set of all ergodic 2-fold self-joinings of $T$.
Let $C(T)$ stand for the centralizer of $T$, i.e.\ the set of all $\mu$-preserving invertible transformations of $X$ commuting with $T_g$ for each $g\in \Gamma$.
Given a transformation $S\in C(T)$, we denote by $\mu_S$ the corresponding \emph{off-diagonal measure} on $X\times X$ defined by $\mu_S(A\times B) := \mu(A\cap S^{-1}B)$ for all $A,B\in\mathfrak B$.
In other words, $\mu_S$ is the image of the measure $\mu$ under the map $x\mapsto(x,Sx)$.
Of course, $\mu_S\in J_2^e(T)$ for every $S\in C(T)$.
If $T$ is weakly mixing, $\mu\times\mu$ is also an ergodic self-joining.
If $J_2^e(T) \subset \{\mu_S\mid S\in C(T)\}\cup\{\mu\times\mu\}$ then $T$ is called \emph{2-fold simple} \cite{Ve}, \cite{dJR}.
By a \emph{factor} of $T$ we mean a non-trivial proper $T$-invariant sub-$\sigma$-algebra of $\mathfrak B$.
If $T$ has no non-trivial proper factors then $T$ is called \emph{prime}.
In \cite{Ve} it was shown that if $T$ is 2-fold simple then for each non-trivial factor $\mathfrak{F}$ of $T$ there exists a compact (in the strong operator topology) subgroup $K_{\fr F}\subset C(T)$ such that $\fr F=\fr F_{K_{\fr F}}$, where
$$
\mathfrak{F}_K = \{ A\in\mathfrak B \mid \mu(kA\bigtriangleup A)=0 \text{ for all } k\in K \}
$$
is the fixed algebra of $K$.
In particular, $\mathfrak{F}$ (or, more precisely, the restriction of $T$ to $\mathfrak{F}$) is prime if and only if $K_\mathfrak{F}$ is a maximal compact subgroup of $C(T)$.

One of the natural questions arising after the general theory of simple actions was developed in \cite{dJR} is: are there simple maps with non-unique prime factors?
The first example of such maps was constructed by Glasner and Weiss \cite{GlW} as an inverse limit of certain horocycle flows.
For that they used some subtle facts from Ratner's theory of joinings for horocycle flows and properties of lattices in $SL_2(\Bbb R)$.
The authors of a later paper \cite{DdJ} utilized a more elementary cutting-and-stacking technique  to construct a weakly mixing 2-fold simple transformation which has countably many factors, all of which are prime.
Our purpose in the present paper is to use a similar cutting-and-stacking technique to produce a {\it mixing} transformation which has {\it uncountably} many factors, all of which are prime.

Via $(C,F)$-construction we produce a measure preserving action $T$ of an auxiliary group $G=\mathbb{Z}\times(\mathbb{R}\rtimes\Z_2)$ such that the transformation $T_{(1,0,0)}$ is mixing 2-fold simple and $C(T_{(1,0,0)}) = \{T_g\mid g\in G\}$.
Since all non-trivial compact subgroups of $G$ are $G_b=\{(0,0,0),(0,b,1)\}$, $0\neq b\in\R$, and all of them are maximal, this gives an example of 2-fold simple transformation with uncountably many prime factors.
All these factors are 2-to-1 and pairwise isomorphic.

We also correct a gap in the proof of \cite[Lemma~2.3(ii)]{DdJ} (see Remark~\ref{gap_rem}).

The skeleton of the proof of the main result is basically the same as in \cite{DdJ}, where the ``discrete case'' (i.e.\ the auxiliary group is discrete) was under consideration.
To work with the $(C,F)$-construction for actions of continuous (i.e.\ non-discrete) groups we use the approximation techniques from \cite{Da_H}.

\section{$(C,F)$-construction}

We now briefly outline the $(C,F)$-construction of measure preserving actions for locally compact groups.
For details see \cite{Da_CF} and references therein.

Let $G$ be a unimodular locally compact second countable (l.c.s.c.)\ amenable group.
Fix a ($\sigma$-finite) Haar measure $\lambda$ on it.
Given two subsets $E,F\subset G$, by $EF$ we mean their algebraic product, i.e. $EF=\{ef \mid e\in E, f\in F\}$.
The set $\{e^{-1} \mid e\in E\}$ is denoted by $E^{-1}$.
If $E$ is a singleton, say $E=\{e\}$, then we will write $eF$ for $EF$.
Given a finite set $A$, $|A|$ will denote the cardinality of $A$.
Given a subset $F\subset G$ of finite Haar measure, $\lambda_F$ will denote the probability on $F$ given by $\lambda_F(A) := \lambda(A)/\lambda(F)$ for each measurable $A\subset F$.
If $D$ is finite, then $\kappa_D$ will denote the equidistributed probability on $D$, that is $\kappa_D(A) := |A|/|D|$ for each $A\subset D$.

To define a $(C,F)$-action of $G$ we need two sequences $(F_n)_{n=0}^\infty$ and $(C_n)_{n=1}^\infty$ of subsets in $G$ such that the following are satisfied:
\begin{align}
& (F_n)_{n=0}^\infty \text{ is a F{\o}lner sequence in $G$}, \label{CF1}\\
& C_n \text{ is finite and } |C_n| > 1,  \label{CF2}\\
& F_n C_{n+1} \subset F_{n+1},  \label{CF3}\\
& F_n c \cap F_n c' = \varnothing \text{ for all } c\neq c'\in C_{n+1}.  \label{CF4}
\end{align}
This means that $F_n C_{n+1}$ consists of $|C_{n+1}|$ mutually disjoint `copies' $F_n c$,
$c\in C_{n+1}$, of $F_n$ and all these copies are contained in $F_{n+1}$.

First, we define a probability space $(X,\mu)$ in the following way.
We equip $F_n$ with the measure $(|C_1|\cdots|C_n|)^{-1}\lambda\upharpoonright F_n$ and endow $C_n$ with the equidistributed probability measure.
Let $X_n := F_n \times \prod_{k>n} C_k$ stand for the product of measure spaces.
Define an embedding $X_n \to X_{n+1}$ by setting
$$
(f_n,c_{n+1},c_{n+2},\ldots) \mapsto (f_n c_{n+1},c_{n+2},\ldots).
$$
It is easy to see that this embedding is measure preserving.
Then $X_1\subset X_2\subset\cdots$.
Let $X:=\bigcup_{n=0}^\infty X_n$ denote the inductive limit of the sequence of measure spaces $X_n$ and let $\goth B$ and $\mu$ denote the corresponding Borel $\sigma$-algebra and measure on $X$ respectively.
Then $X$ is a standard Borel space and $\mu$ is $\sigma$-finite.
It is easy to check that $\mu$ is finite if and only if
\begin{equation} \label{fin}
\lim_{n\to\infty}\frac{\lambda(F_n)}{|C_1|\cdots|C_n|} < \infty.
\end{equation}
If (\ref{fin}) is satisfied then we choose (i.e., normalize) $\lambda$ in such a way that $\mu(X)=1$.

Now we define a $\mu$-preserving action of $G$ on $X$.
Suppose that the following is satisfied:
\begin{equation} \label{CFgap}
\text{for any $g\in G$, there is $m\geqslant 0$ with $gF_n C_{n+1}\subset F_{n+1}$ for all $n\geqslant m$.}
\end{equation}
For such $n$, take $x\in X_n\subset X$ and write the expansion $x=(f_n,c_{n+1},c_{n+2},\ldots)$ with $f_n\in F_n$ and $c_i\in C_i$, $i>n$.
Then we let
$$
T_g x := (gf_nc_{n+1},c_{n+2},\ldots)\in X_{n+1}\subset X.
$$
It follows from (\ref{CFgap}) that $T_g$ is a well defined $\mu$-preserving transformation of $X$.
Moreover, $T_gT_h=T_{gh}$, i.e. $T:=(T_g)_{g\in G}$ is a $\mu$-preserving Borel action of $G$ on $X$.
$T$ is called the \emph{$(C,F)$-action of $G$ associated with $(C_{n+1},F_n)_{n=0}^\infty$}.

We now recall some basic properties of $(X, \goth B, \mu, T)$.
Given a Borel subset $A\subset F_n$, we put
$$
[A]_n := \{x\in X \mid x=(f_n,c_{n+1},c_{n+2},\ldots)\in X_n \text{ and } f_n\in A\}
$$
and call this set an \emph{$n$-cylinder}.
It is clear that the $\sigma$-algebra $\goth B$ is generated by the family of all cylinders.
Given Borel subsets $A,B\subset F_n$, we have
\begin{align}
& [A\cap B]_n = [A]_n \cap [B]_n, [A\cup B]_n = [A]_n \cup [B]_n, \\
& [A]_n = [AC_{n+1}]_{n+1} = \bigsqcup_{c\in C_{n+1}} [Ac]_{n+1}, \\
& \mu([A]_n) = |C_{n+1}| \mu([Ac]_{n+1}) \text{ for every } c\in C_{n+1}, \\
& \mu([A]_n) = \mu(X_n) \lambda_{F_n}(A), \\
& T_g[A]_n = [gA]_n \text{ if } gA\subset F_n, \\
& T_g[A]_n = T_h[h^{-1}gA]_n \text{ if } h^{-1}gA\subset F_n. \label{CFhg}
\end{align}
Each $(C,F)$-action is of funny rank one (for the definition see \cite{Fe} for the case of $\Z$-actions and \cite{So} for the general case) and hence ergodic.
It also follows from (\ref{CF2}) that $T$ is conservative.

\section{Main result}

By $\Z_n$ we denote a cyclic group of order $n$, i.e. $\Z_n=\Z/n\Z=\{0,1,\ldots,n-1\}$.
Let $G:=\Z\times(\R\rtimes\Z_2)$ with multiplication law as follows
$$
(x,a,n)(y,b,m) := (x+y,a+(-1)^n b,n+m).
$$
Then the center $C(G)$ of $G$ is $\Z\times\{0\}\times\{0\}$.
Each compact subgroup of $G$ coincides with $G_b=\{(0,0,0),(0,b,1)\}$ for some $b\in\R$.
Notice that $G_b$ is a maximal compact subgroup of $G$ if $b\neq 0$.

To construct the required $(C,F)$-action of $G$ we will determine a sequence $(C_{n+1},F_n)_{n=0}^\infty$.
Let $(r_n)_{n=0}^\infty$ be an increasing sequence of positive integers such that
\begin{equation}
\lim_{n\to\infty}\frac{n^4}{r_n}=0.
\end{equation}
Below --- just after Lemma~\ref{dJlemma} --- one more restriction on the growth of $(r_n)_{n=0}^\infty$ will be imposed and we will assume that $r_n$ is large so that (\ref{dJlem}) is satisfied.
We define recurrently three other sequences $(\widetilde{a}_n)_{n=0}^\infty$, $(a_n)_{n=1}^\infty$ and $(b_n)_{n=1}^\infty$ of positive integers by setting
\begin{align*}
\widetilde{a}_0 &:= 1, \\
a_n &:= (2r_{n-1}+1)\widetilde{a}_{n-1} \text{ for $n\geqslant 1$}, \\
b_n &:= (2n-1)\widetilde{a}_{n-1} \text{ for $n\geqslant 1$}, \\
\widetilde{a}_n &:= a_n + b_n + n \text{ for $n\geqslant 1$}.
\end{align*}
For each $n\in\N$, we let
\begin{align*}
I_n &:= \{-n,\ldots,n\}^2\subset\Z^2, \\
H_n &:= \{-r_n,\ldots,r_n\}^2\subset\Z^2, \\
F_n &:= (-a_n,a_n]_\Z \times (-a_n,a_n]_\R \times \Z_2 , \\
S_n &:= (-b_n,b_n]_\Z \times (-b_n,b_n]_\R \times \Z_2 , \\
\widetilde{F}_n &:= (-\widetilde{a}_n,\widetilde{a}_n]_\Z \times (-\widetilde{a}_n,\widetilde{a}_n]_\R \times \Z_2 .
\end{align*}
We also consider a homomorphism $\phi_n\colon \Z^2 \to G$ given by
$$
\phi_n(i,j) := (2i\widetilde{a}_n, 2j\widetilde{a}_n, 0).
$$
We then have
\begin{equation}
S_n\subset F_n, \quad F_n S_n = S_n F_n \subset \widetilde{F}_n \subset G,
\end{equation}
\begin{equation}
S_{n+1} = \widetilde{F}_n \phi_n(I_n) = \bigsqcup_{h\in I_n} \widetilde{F}_n \phi_n(h) = \bigsqcup_{h\in I_n} \phi_n(h) \widetilde{F}_n,
\end{equation}
\begin{equation}
F_{n+1} = \widetilde{F}_n \phi_n(H_n) = \bigsqcup_{h\in H_n} \widetilde{F}_n \phi_n(h) = \bigsqcup_{h\in H_n} \phi_n(h) \widetilde{F}_n,
\end{equation}

Suppose also that $F_n$ is equipped with a finite partition $\xi_n$ such that the following are satisfied:
\begin{enumerate}
\item[(i)] the diameter of each atom of $\xi_n$ is less than $\tfrac 1 n$,
\item[(ii)] for each atom $A\in\xi_{n-1}$ and each element $c\in C_n$, the subset $Ac\subset F_n$ is $\xi_n$-measurable and
\item[(iii)] $\xi_n$ is symmetric, that is $A^{-1}\in\xi_n$ whenever $A\in\xi_n$.
\end{enumerate}
It follows that for each measurable subset $A\subset F_n$, any $\varepsilon>0$ and for all $k$ large enough, there is a $\xi_k$-measurable subset $B\subset F_k$ such that $\mu([A]_n \bigtriangleup [B]_k)<\varepsilon$.
We will denote by $\sigma(\xi_n)$ the $\sigma$-algebra on $F_n$ generated by $\xi_n$.

For a finite subset $D$ in $S_n$, we denote by $\kappa_D$ the corresponding normalized Dirac comb,
i.e.\ a measure on $S_n$ given by $\kappa_D(A):=\dfrac{|A\cap D|}{|D|}$ for each subset $A\subset S_n$.
Given two subsets $A,B\subset F_n$ define a function $f_{A,B}\colon S_n\times S_n \to \R$ by setting $f_{A,B}(x,y) := \dfrac{\lambda(Ax\cap By)}{\lambda(F_n)}$, $x,y\in S_n$.
Choose a finite subset $D_n$ in $S_n$ such that
\begin{equation} \label{discr_meas}
\left| \int_{S_n\times S_n} f_{Ag,Bh} d\kappa_{D_n}d\kappa_{D_n} - \frac{1}{\lambda(S_n)^2}\int_{S_n\times S_n} f_{Ag,Bh} d\lambda d\lambda \right| < \frac{1}{n}
\end{equation}
for each $\xi_n$-measurable subsets $A,B\subset F_n$ and any $g,h\in F_n$ with $AgS_n,BhS_n\subset F_n$.
For instance, let $\xi_n$ consists of `rectangles' $\{a\}\times \Delta\times \{m\} \subset G$, where $a\in (-a_n,a_n]_{\Z}$, $m\in\Z_2$ and $\Delta\subset (-a_n,a_n]_{\R}$ is a subinterval of length $n^{-1}$, and set $D_n:=\{ (a,kn^{-2},m) \mid a\in(-b_n,b_n]_{\Z}, k\in (-n^2 b_n,n^2 b_n]_{\Z}, m\in\Z_2 \}$.
It is an easy exercise to check that (\ref{discr_meas}) is satisfied for such $\xi_n$ and $D_n$.
We notice also that $|D_n^0| = |D_n^1|$.

Given a finite (signed) measure $\nu$ on a finite set $D$, we let $\|\nu\|_1 := \sum_{d\in D}|\nu(d)|$.
If $\pi\colon D\to E$ then clearly $\|\nu\circ\pi^{-1}\|_1 \leqslant \|\nu\|_1$.
Given a finite set $Y$ and a mapping $s\colon Y\to D$, let $\dist_{y\in Y}s(y)$ denote the image of the equidistribution on $Y$ under $s$:
$$
\dist_{y\in Y}s(y) := \frac{1}{|Y|}\sum_{y\in Y} \delta_{s(y)} = \kappa_D\circ s^{-1}.
$$

The following lemma easily follows from \cite[Lemma~2.1]{dJ} (cf. \cite[Lemma~3.2]{Da_H}).
\begin{lemma}\label{dJlemma}
Let $D$ be a finite set.
Then given $\varepsilon>0$ and $\delta>0$, there is $R\in\N$ such that for each $r>R$, there exists a map $s\colon \{-r,\ldots,r\}^2 \to D$ such that
$$
\left\| \dist_{0\leqslant t <N} (s_n(h+(t,0)),s_n(h'+(t,0))) - \kappa_D\times \kappa_D \right\|_1 < \varepsilon
$$
for each $N>\delta r$ and $h\neq h'\in \{-r,\ldots,r\}^2$ with $h_1+N<r$ and $h'_1+N<r$.\footnote{Here and below by $a\neq b\in A$ we denote two different elements $a$, $b$ of a set $A$.}
\end{lemma}

Applying this Lemma with $\varepsilon=\frac{1}{n}$ and $\delta = \frac{1}{n^2}$ we get the following.
If $r_n$ is large enough then there is a mapping $s_n\colon H_n \to D_n$ such that for any $N>\dfrac{r_n}{n^2}$ and $h\neq h'\in H_n\cap (H_n - (N-1,0))$ we have
\begin{equation} \label{dJlem}
\left\| \dist_{0\leqslant t <N} (s_n(h+(t,0)),s_n(h'+(t,0))) - \kappa_{D_n}\times \kappa_{D_n} \right\|_1 < \frac 1 n.
\end{equation}
From now on we will assume that $r_n$ is large so that this condition is satisfied and for each $n$ fix $s_n\colon H_n \to D_n$ satisfying (\ref{dJlem}).

Now we define a map $c_{n+1}\colon H_n \to G$ by setting $c_{n+1}(h) := s_n(h)\phi_n(h)$.
We put $C_{n+1} := c_{n+1}(H_n)$.

The reader should have the following picture in mind.
The set $F_{n+1}$ is exactly tiled with the sets $\widetilde{F}_n \phi_n(h)$, $h\in H_n$, which may be thought of as `windows'.
Each $F_n$ has a `natural' translate $F_n \phi_n(h)$ in $\widetilde{F}_n \phi_n(h)$ but the translate we actually choose is the natural translate perturbed by a further translation $s_n(h)$ which is chosen in a `random' way and does not move $F_n \phi_n(h)$ out of its window.

It is easy to derive that (\ref{CF1})--(\ref{CFgap}) are satisfied for the sequence $(F_n,C_{n+1})_{n=0}^\infty$.
Hence the associated $(C,F)$-action $T=(T_g)_{g\in G}$ of $G$ is well defined on a standard probability space $(X,\fr B,\mu)$.

We now state the main result.

\begin{theorem} \label{MainTh}
The transformation $T_{(1,0,0)}$ is mixing and 2-fold simple.
All non-trivial proper factors of $T_{(1,0,0)}$ are of the form $\fr F_{G_b}$, $0\neq b\in\R$.
All these factors are 2-to-1, prime and pairwise isomorphic.
\end{theorem}

We first prove some technical lemmata.
After that in Proposition~\ref{mixingProp} we show mixing for $T_{(1,0,0)}$ and in Proposition~\ref{simpl} we prove simplicity and describe the centralizer of $T_{(1,0,0)}$.
The structure of factors follows then from Veech's theorem.

Denote by $G^0$ the subgroup $\Z\times\R\times\{0\}$ of index 2 in $G$.
Given any subset $A$ in $G$ we set $A^0 := A\cap G^0$ and $A^1 := A\setminus A^0$.
We will refer to $A^0$ and $A^1$ as `\emph{levels}' of $A$.
We will say that a subset $A\subset G$ is \emph{$\varepsilon$-balanced} if $$|\lambda(A^0) - \lambda(A^1)| < \varepsilon \lambda(A).$$
Denote by $\pi_3\colon G\to\Z_2 $ a natural projection on the third coordinate.
Since $\kappa_{D_n}\circ\pi_3^{-1} = \kappa_{\Z_2 }$, it follows from (\ref{dJlem}) that
\begin{equation} \label{balanced}
\| \dist_{h\in H_n} \pi_3\circ s_n(h) - \kappa_{\Z_2 } \|_1 < \frac{1}{n}.
\end{equation}
In particular, for any $A^* \subset F_n$ the set $A=A^*C_{n+1}$ is $\frac 1 n$-balanced:
\begin{equation} \label{balanced}
|\lambda(A^0) - \lambda(A^1)| < \frac{1}{n} \lambda(A).
\end{equation}
Indeed, since $$A^0 = \bigsqcup_{h\in s_n^{-1}(G^0)} A^{*0} c_n(h) \sqcup \bigsqcup_{h\in s_n^{-1}(G^1)} A^{*1} c_n(h),$$
$$\lambda(A^0) = \lambda(A^{*0})|s_n^{-1}(G^0)| + \lambda(A^{*1})|s_n^{-1}(G^1)|,$$
and similarly
$$\lambda(A^1) = \lambda(A^{*1})|s_n^{-1}(G^0)| + \lambda(A^{*0})|s_n^{-1}(G^1)|.$$
Hence
\begin{align*}
|\lambda(A^0) - \lambda(A^1)| &= |\lambda(A^{*0}) - \lambda(A^{*1})| \left| |s_n^{-1}(G^0)| - |s_n^{-1}(G^1)| \right| \\
&\leqslant \frac{1}{|H_n|} \lambda(A) \left| |s_n^{-1}(G^0)| - |s_n^{-1}(G^1)| \right|.
\end{align*}
It remains to notice that
\begin{align*}
\frac{1}{|H_n|} \left| |s_n^{-1}(G^0)| - |s_n^{-1}(G^1)| \right| &\leqslant \left| \frac{|s_n^{-1}(G^0)|}{|H_n|} - \frac{1}{2} \right| + \left| \frac{|s_n^{-1}(G^1)|}{|H_n|} - \frac{1}{2} \right| \\
&= \| \dist_{h\in H_n} \pi_3\circ s_n(h) - \kappa_{\Z_2 } \|_1 < \frac{1}{n}
\end{align*}
by (\ref{balanced}).
It follows that $A=A^*C_{n+1}$ is $\frac 1 n$-balanced for each $A^* \subset F_n$.

Given $h=(h_1,h_2)\in\Z^2$, we let $h^* := (h_1,-h_2)$.

\begin{lemma} \label{mainLem}
Let $f=f'\phi_{n-1}(h)$ with $f'\in\widetilde{F}_{n-1}$ and $h\in\Z^2$.
\begin{enumerate}
\item[(i)] 
Suppose $f\in G^\alpha$ and let $\beta:=1-\alpha$.
Let
\begin{align*}
L^-_n &:= \widetilde{F}_{n-1}^\alpha \phi_{n-1}(I_{n-2} + h) \sqcup \widetilde{F}_{n-1}^\beta \phi_{n-1}(I_{n-2} + h^*) \text{ and} \\
L^+_n &:= \widetilde{F}_{n-1}^\alpha \phi_{n-1}(I_{n} + h) \sqcup \widetilde{F}_{n-1}^\beta \phi_{n-1}(I_{n} + h^*).
\end{align*}
Then $L^-_n \subset fS_n \subset L^+_n$.
Hence
$$
\frac{\lambda(fS_n\bigtriangleup L^-_n)}{\lambda(S_n)} = \overline{o}(1).
$$

\item[(ii)] 
If, in addition, $fS_n\subset F_n$ then for any subset $A=A^*C_{n-1}$ with $A^*\subset F_{n-2}$ we have
$$
\frac{\lambda(AC_n\cap fS_n)}{\lambda(S_n)} = \lambda_{F_{n-1}}(A) + \overline{o}(1).
$$
\end{enumerate}
\end{lemma}

Here $\overline{o}(1)$ means a sequence that goes to 0 as $n\to\infty$ and does not depend on the choice of $A^*$ in $F_{n-2}$.

\begin{proof}
(i) Suppose $f\in G^0$ (the case $f\in G^1$ is considered in a similar way).
We have
\begin{align*}
fS_n &= f'\phi_{n-1}(h) \widetilde{F}_{n-1} \phi_{n-1}(I_{n-1}) \\
&= f'\widetilde{F}_{n-1}^0 \phi_{n-1}(h+I_{n-1}) \sqcup f'\widetilde{F}_{n-1}^1 \phi_{n-1}(h^*+I_{n-1}).
\end{align*}
Since $\widetilde{F}_{n-1}^0 \widetilde{F}_{n-1}^\alpha \subset \bigsqcup_{u\in I_1} \widetilde{F}_{n-1}^\alpha \phi_{n-1}(u)$, there exists a partition of $\widetilde{F}_{n-1}^\alpha$ into subsets $A_u^\alpha$, $u\in I_1$, such that $f'A_u^\alpha \subset \widetilde{F}_{n-1}^\alpha \phi_{n-1}(u)$ for any $u$ and $\alpha=0,1$.
Therefore
$$
fS_n = \bigsqcup_{u\in I_1} (f'A_u^0 \phi_{n-1}(u)^{-1}\phi_{n-1}(u+h+I_{n-1}) \sqcup f'A_u^1 \phi_{n-1}(u)^{-1}\phi_{n-1}(u+h^*+I_{n-1})).
$$
It remains to notice that $\bigsqcup_{u\in I_1} f'A_u^\alpha \phi_{n-1}(u)^{-1} = \widetilde{F}_{n-1}^\alpha$.

(ii) Since $fS_n \subset F_n$ and $F_n = \widetilde{F}_{n-1} \phi_{n-1}(H_{n-1})$, it follows from (i) that the subsets $K:=I_{n-1} + h$ and $K^*:=I_{n-1} + h^*$ are contained in $H_{n-1}$.
Therefore 
\begin{align*}
&\frac{\lambda(AC_n\cap fS_n)}{\lambda(S_n)}
= \sum_{k\in H_{n-1}} \frac{\lambda(Ac_n(k)\cap fS_n)}{\lambda(S_n)} \\
&= \sum_{k\in H_{n-1}} \frac{\lambda(Ac_n(k)\cap L^-_n)}{\lambda(S_n)} + \overline{o}(1) \\
&= \frac{1}{\lambda(S_n)} \sum_{k\in H_{n-1}} \lambda(As_{n-1}(k)\phi_{n-1}(k)\cap \widetilde{F}_{n-1}^\alpha \phi_{n-1}(K) \sqcup \widetilde{F}_{n-1}^\beta \phi_{n-1}(K^*)) + \overline{o}(1) \\
&= \frac{1}{\lambda(S_n)} \sum_{k\in K} \lambda(As_{n-1}(k) \cap \widetilde{F}_{n-1}^\alpha) +
\frac{1}{\lambda(S_n)} \sum_{k\in K^*} \lambda(As_{n-1}(k) \cap \widetilde{F}_{n-1}^\beta) + \overline{o}(1).
\end{align*}
Notice that
$$\lambda(As_{n-1}(k) \cap \widetilde{F}_{n-1}^{\alpha}) =
\begin{cases}
\lambda(A^\alpha), \text{ if $s_{n-1}(k) \in G^0$;} \\
\lambda(A^\beta), \text{ if $s_{n-1}(k) \in G^1$.}
\end{cases}
$$
In any case, since $A=A'C_{n-1}$ is $\frac{1}{n-2}$-balanced, we conclude from (\ref{balanced}) that $$\lambda(As_{n-1}(k) \cap \widetilde{F}_{n-1}^{\alpha}) = (\frac 1 2 + \overline{o}(1)) \lambda(A).$$
In a similar way $$\lambda(As_{n-1}(k) \cap \widetilde{F}_{n-1}^{\beta}) = (\frac 1 2 + \overline{o}(1)) \lambda(A).$$
Hence
\begin{multline*}
\frac{\lambda(AC_n\cap fS_n)}{\lambda(S_n)}
= \frac{\lambda(A)|K|(1+\overline{o}(1))}{\lambda(S_n)}+ \overline{o}(1) \\
= \frac{\lambda(A)}{\lambda(F_{n-1})} \cdot \frac{\lambda(F_{n-1}) |K|}{\lambda(S_n)} \cdot (1+\overline{o}(1)) + \overline{o}(1) \\
= \lambda_{F_{n-1}}(A) \cdot \frac{\lambda(F_{n-1}) (2n-1)^2}{(2n+1)^2 \lambda(\widetilde{F}_{n-1})} \cdot (1+\overline{o}(1)) + \overline{o}(1)
= \lambda_{F_{n-1}}(A) + \overline{o}(1).
\end{multline*}
\end{proof}

\begin{remark} \label{gap_rem}
We note that there is a gap in \cite[Lemma~2.3(ii)]{DdJ}.
It was stated there that the claim (ii) is true for each subset
$A\subset F_{n-1}$.
This is not true.
However --- as was shown in Lemma~\ref{mainLem}(ii) above --- the
claim is true if $A=A^*C_{n-1}$ for an arbitrary subset $A^*\subset
F_{n-2}$.
This corrected  version of the claim suffices to apply it in the proof
of \cite[Theorem~2.5]{DdJ} which is the only place in that paper where
\cite[Lemma~2.3(ii)]{DdJ} was used.
\end{remark}

We will also use the following simple lemma.

\begin{lemma} \label{techLem}
Let $A$, $B$ and $S$ be subsets of finite Haar measure in $G$.
Then
$$
\int_{S\times S} \lambda(Ax\cap By) d\lambda(x)d\lambda(y) = \int_{A\times B} \lambda(aS\cap bS) d\lambda(a)d\lambda(b).
$$
\end{lemma}

\begin{proof}
Notice that $G$ is unimodular.
Consider two subsets in $G^3$:
\begin{align*}
\Omega_1 :=& \{ (a,x,y) \mid x\in S, y\in S, a\in A\cap Byx^{-1} \} \\
=& \{ (a,x,y) \mid a\in A, y\in S, x\in a^{-1}By\cap S \} \text{ and} \\
\Omega_2 :=& \{ (a,b,y) \mid a\in A, b\in B, y\in b^{-1}aS\cap S \} \\
=& \{ (a,b,y) \mid a\in A, y\in S, b\in B\cap aSy^{-1} \}.
\end{align*}
It is clear that the maping $\Omega_1\ni(a,x,y)\mapsto (a,axy^{-1},y)\in\Omega_2$ is 1-to-1 and $\lambda^3$-preserving.
Applying Fubini theorem we obtain that
$$
\int_{S\times S} \lambda(Ax\cap By) d\lambda(x)d\lambda(y) = \lambda^3(\Omega_1) = \lambda^3(\Omega_2) = \int_{A\times B} \lambda(aS\cap bS) d\lambda(a)d\lambda(b).
$$
\end{proof}

The following lemma is the first step to prove mixing for $T_{(1,0,0)}$.
Let $h_0 := (1,0) \in \Z^2$.
Then $\phi_n(h_0) = (1,0,0)^{2\widetilde{a}_n}$.

\begin{lemma} \label{LemWM}
Given a sequence of subsets $H_n^*\subset H_n$ such that $\frac{|H_n^*|}{|H_n|}\to\delta$ for some $\delta\geqslant 0$, we let $C_n^*:=c_n(H_{n-1}^*)$.
Then
\begin{equation} \label{WMlem}
\sup_{A^*,B^*\in\sigma(\xi_{n-1})} \left| \mu(T_{\phi_n(h_0)} [A^*C_n^*]_n \cap [B^*]_{n-1}) - \mu([A^*C_n^*]_n) \mu([B^*]_{n-1}) \right| \to 0.
\end{equation}
\end{lemma}

\begin{proof}
Let $A,B\subset F_n$ be $\xi_n$-measurable.
We set $F^\circ_n := \{f\in F_n\mid fS_nS_n^{-1}\subset F_n\},$
$A^\circ:= A\cap F^\circ_n$, $B^\circ:=B\cap F^\circ_n$, $H'_n:=H_n\cap (H_n - h_0)$.
It is clear that $\mu(F_n\setminus F^\circ_n)\to 0$ and $\frac{|H'_n|}{|H_n|}\to 1$ as $n\to\infty$.
Since $\phi_n(h_0)\in C(G)$ for all $n\in\N$, we have
$$
\phi_n(h_0) A c_{n+1}(h) = A s_n(h) \phi_n(h_0 + h) = A s_n(h) s_n(h_0 + h)^{-1} c_{n+1} (h_0 + h).
$$
whenever $h\in H'_n$.
In particular, $\phi_n(h_0) A^\circ c_{n+1}(h)\subset F_{n+1}$ for all $h\in H'_n$.
Then
\begin{align*}
\mu(T_{\phi_n(h_0)}[A]_n \cap &[B]_n)
= \mu(T_{\phi_n(h_0)}[A^\circ]_n \cap [B^\circ]_n) + \overline{o}(1) \\
&= \sum_{h\in H_n} \mu(T_{\phi_n(h_0)}[A^\circ c_{n+1}(h)]_{n+1} \cap [B^\circ]_n) + \overline{o}(1) \\
&= \sum_{h\in H'_n} \mu(T_{\phi_n(h_0)}[A^\circ c_{n+1}(h)]_{n+1} \cap [B^\circ]_n) + \overline{o}(1) \\
&= \sum_{h\in H'_n} \mu([A^\circ s_n(h) s_n(h_0 + h)^{-1} c_{n+1} (h_0 + h)]_{n+1} \cap [B^\circ]_n) + \overline{o}(1) \\
&= \sum_{h\in H'_n} \mu([(A^\circ s_n(h) s_n(h_0 + h)^{-1} \cap B^\circ) c_{n+1} (h_0 + h)]_{n+1}) + \overline{o}(1) \\
&= \frac{1}{|H_n|}\sum_{h\in H'_n} \mu([A^\circ s_n(h) s_n(h_0 + h)^{-1} \cap B^\circ]_n) + \overline{o}(1) \\
&= \frac{1}{|H_n|}\sum_{h\in H'_n} \lambda_{F_n}(A^\circ s_n(h) \cap B^\circ s_n(h_0 + h))\mu(X_n) + \overline{o}(1) \\
&= \frac{1}{|H'_n|}\sum_{h\in H'_n} \lambda_{F_n}(A^\circ s_n(h) \cap B^\circ s_n(h_0 + h)) + \overline{o}(1) \\
&= \frac{1}{|H'_n|}\sum_{h\in H'_n} \lambda_{F_n}(As_n(h) \cap B s_n(h_0 + h)) + \overline{o}(1).
\end{align*}
Let $\nu_n := \dist_{h\in H'_n}(s_n(h),s_n(h+h_0))$.
Set $f_{A,B}(x,y) := \lambda_{F_n}(Ax\cap By) = \frac{\lambda(Ax\cap By)}{\lambda(F_n)}$.
Notice that
$$
\nu_n = \frac{1}{2r_n-1} \sum_{i=-r_n}^{r_n} \dist_{-r_n \leqslant t < r_n}(s_n(t,i),s_n(t+1,i)).
$$
It follows from (\ref{dJlem}) that $\| \nu_n - \kappa_{D_n}\times \kappa_{D_n} \|_1 < \frac{1}{n}$.
Then by (\ref{discr_meas})
\begin{multline*}
\mu(T_{\phi_n(h_0)}[A]_n \cap [B]_n) = \int_{S_n\times S_n} f_{A,B}d\nu_n + \overline{o}(1) \\
= \int_{S_n\times S_n} f_{A,B}d\kappa_{D_n}d\kappa_{D_n} + \overline{o}(1)
= \frac{1}{\lambda(S_n)^2} \int_{S_n\times S_n} f_{A,B}d\lambda d\lambda + \overline{o}(1),
\end{multline*}
Now take $A:=A^* C^*_n$ and $B:=B^* C_n$ for some $\xi_{n-1}$-measurable subsets $A^*,B^*\subset F_{n-1}$.
We say that elements $c$ and $c'$ of $C_n$ are \emph{partners} if $F_{n-1} c S_n \cap F_{n-1} c' S_n \neq \varnothing$.
We then write $c \bowtie c'$.
Since $A^*c x \cap B^*c' y = \varnothing$ for $c \not\bowtie c'$, it follows that
\begin{align*}
\int_{S_n\times S_n} f_{A,B}d\lambda d\lambda &=
\int_{S_n\times S_n} \lambda_{F_n} (A^*C_n x \cap B^*C_n y) d\lambda(x) d\lambda(y) \\
&= \frac{1}{\lambda(F_n)} \int_{S_n\times S_n} \sum_{C^*_n\ni c\bowtie c'\in C_n} \lambda (A^*c x \cap B^*c' y) d\lambda(x) d\lambda(y).
\end{align*}
Applying Lemma~\ref{techLem} we now obtain that
$$
\int_{S_n\times S_n} f_{A,B}d\lambda d\lambda =
\frac{1}{\lambda(F_n)} \sum_{C^*_n\ni c\bowtie c'\in C_n} \int_{A^* \times B^*} \lambda (ac S_n \cap bc' S_n) d\lambda(a) d\lambda(b).
$$
Next, we note that
$$
|\lambda (ac S_n \cap bc' S_n) - \lambda (c S_n \cap c' S_n)| \leqslant 8n\lambda(\widetilde{F}_{n-1}) = \overline{o}(1) \lambda(S_n).
$$
Each $c\in C_n$ has no more than $2(4n+1)^2$ partners.
Therefore
\begin{align*}
&\mu(T_{\phi_n(h_0)}[A^*C^*_n]_n \cap [B^*]_{n-1}) \\
&= \frac{1}{\lambda(S_n)^2} \sum_{C^*_n\ni c\bowtie c'\in C_n} \int_{A^* \times B^*} \frac{\lambda (c S_n \cap c' S_n) + \lambda(S_n)\overline{o}(1)}{\lambda(F_n)} d\lambda(a) d\lambda(b) + \overline{o}(1) \\
&= \frac{\lambda(A^*)\lambda(B^*)}{\lambda(F_{n-1})^2}
\frac{\lambda(F_{n-1})^2}{\lambda(S_n)^2 \lambda(F_n)}
\sum_{C^*_n\ni c\bowtie c'\in C_n} \left( \lambda (c S_n \cap c' S_n) + \lambda(S_n)\overline{o}(1) \right) + \overline{o}(1) \\
&= \lambda_{F_{n-1}}(A^*) \lambda_{F_{n-1}}(B^*) \theta_n
\pm \frac{\lambda(F_{n-1})^2 |H^*_n| 2(4n+1)^2 \lambda(S_n) \overline{o}(1)}{\lambda(S_n)^2 \lambda(F_n)}
+ \overline{o}(1) \\
&= \lambda_{F_{n-1}}(A^*) \lambda_{F_{n-1}}(B^*) \theta_n
\pm \frac{\lambda(F_{n-1})^2 |H^*_n| 2(4n+1)^2 \overline{o}(1)}{\lambda(\widetilde{F}_{n-1})^2 (2n-1)^2 |H_n|}
+ \overline{o}(1) \\
&= \lambda_{F_{n-1}}(A^*) \lambda_{F_{n-1}}(B^*) \theta_n
+ \overline{o}(1),
\end{align*}
where $\theta_n = \frac{\lambda(F_{n-1})^2}{\lambda(S_n)^2 \lambda(F_n)} \sum_{C^*_n\ni c\bowtie c'\in C_n} \lambda (c S_n \cap c' S_n)$.
Substituting $A^* = B^* = F_{n-1}$ and passing to the limit we obtain that $\theta_n\to \delta$ as $n\to\infty$.
Hence
$$
\mu(T_{\phi_n(h_0)}[A^*C^*_n]_n \cap [B^*]_{n-1}) = \mu([A^*C^*_n]_n)\mu([B^*]_{n-1}) + \overline{o}(1).
$$
Since $\overline{o}(1)$ does not depend on the choice of $A^*$ and $B^*$ inside $F_{n-1}$, the claim is proven.
\end{proof}

\begin{cor}
The transformation $T_{(1,0,0)}$ is weakly mixing.
\end{cor}

\begin{proof}
Substituting $H_n^* := H_n$ to (\ref{WMlem}) we obtain that
$$
\sup_{A^*,B^*\in\sigma(\xi_{n-1})} \left| \mu(T_{\phi_n(h_0)} [A^*]_{n-1} \cap [B^*]_{n-1}) - \mu([A^*]_{n-1}) \mu([B^*]_{n-1}) \right| \to 0.
$$
Since each measurable subset of $X$ can be approximated by $[A^*]_{n-1}$ for large $n$ and $\xi_{n-1}$-measurable subset $A^*\subset F_{n-1}$, it follows that the sequence $(\phi_n(h_0))_{n=1}^\infty$ is mixing for $T$, that is $\mu(T_{\phi_n(h_0)}A\cap B) \to \mu(A)\mu(B)$ for every pair of measurable subsets $A,B\subset X$.
\end{proof}

\begin{prop} \label{mixingProp}
The transformation $T_{(1,0,0)}$ is mixing.
\end{prop}

\begin{proof}
We have to show that
$$
\lim_{n\to\infty} \mu(T_{g_n}A\cap B) = \mu(A)\mu(B)
$$
for any sequence $(g_n)_{n=1}^\infty$ that goes to infinity in $C(G)$ and every pair of measurable subsets $A,B\subset X$.
Let $g_n\in F_{n+1}\setminus F_n$.
It suffices to show that a subsequence of $(g_n)_{n=1}^\infty$ is mixing for $T$.
We write $g_n = f_n \phi_n(h_n)$ for some $f_n\in \widetilde{F}_n\cap C(G)$ and $h_n\in H_n$.
Denote by $z\colon \Z\to C(G)$ the natural embedding $z(x):=(x,0,0)$.
We may assume that $f_n\in z(\Z_+)$ for all $n$
(the case $f_n\in z(\Z_-)$ is considered in a similar way).
Let $H'_n:= H_n\cap (H_n-h_n)$ and $F'_n:= F_n\cap (f_n^{-1}F_n)$.
Passing to a subsequence, if necessary, we also may assume without loss of generality that
$$
\frac{|H'_n|}{|H_n|}\to \delta_1 \quad\text{and}\quad \frac{\lambda(F'_n)}{\lambda(F_n)}\to \delta_2
$$
for some $\delta_1,\delta_2\geqslant 0$.
Partition $H_n$ into three subsets $H_n^1$, $H_n^2$ and $H_n^3$ as follows
\begin{align*}
& H_n^1 := \{ h\in H_n\mid g_nF_nc_{n+1}(h)\subset F_{n+1}\phi_{n+1}(h_0) \}, \\
& H_n^2 := \{ h\in H_n\mid g_nF_nc_{n+1}(h)\subset F_{n+1} \}, \\
& H_n^3 := H_n\setminus (H_n^1\sqcup H_n^2).
\end{align*}
As before $h_0=(1,0)\in\Z^2$.
Let $C^i_{n+1} := \phi_{n+1}(H_n^i)$.
It is clear that $|H_n^3|\leqslant 4(n+1)(2r_n+1)$ and $\left|H_n^2 \bigtriangleup H'_n\right|\leqslant 2r_n+1$.
Since $|H_n| = (2r_n+1)^2$, it follows that
$$
\frac{|H^1_n|}{|H_n|}\to 1-\delta_1, \quad \frac{|H^2_n|}{|H_n|}\to \delta_1, \quad \frac{|H^3_n|}{|H_n|}\to 0.
$$

Take two $\xi_n$-measurable subsets $A,B\subset F_n$.
Since $$\mu([AC^3_{n+1}]_{n+1})=\frac{|C^3_{n+1}|}{|C_{n+1}|} \mu([A]_n)\leqslant \frac{1}{2r_n+1}\to 0,$$ we have
\begin{equation} \label{first}
\left| \mu(T_{g_n}[AC^3_{n+1}]_{n+1} \cap [B]_n) - \mu([AC^3_{n+1}]_{n+1}) \mu([B]_n) \right| \to 0,
\end{equation}
so $[F_nC^3_{n+1}]_{n+1}$ is negligible.
It suffices to show mixing separately on each of the remaining subsets $[F_nC^1_{n+1}]_{n+1}$ and $[F_nC^2_{n+1}]_{n+1}$.

First, we note that $\phi_{n+1}(h_0)^{-1} g_n F_n C^1_{n+1} \subset F_{n+1}$.
Thus, by (\ref{CFhg}),
$$
T_{g_n}[AC^1_{n+1}]_{n+1} = T_{\phi_{n+1}(h_0)} [\phi_{n+1}(h_0)^{-1}g_n AC^1_{n+1}]_{n+1}.
$$
By Lemma~\ref{LemWM} (with $C^*_{n+1}:=\phi_{n+1}(h_0)^{-1}\phi_n(h_n) C^1_{n+1}$ and $A^*:=f_n A$) we obtain that
\begin{equation}
\left| \mu(T_{g_n}[AC^1_{n+1}]_{n+1} \cap [B]_n) - \mu([AC^1_{n+1}]_{n+1}) \mu([B]_n) \right| \to 0.
\end{equation}

It remains to consider the second case involving $C^2_{n+1}$.
If $\delta_1=0$, then obviously
\begin{equation}
\mu([AC^2_{n+1}]_{n+1}) \to 0.
\end{equation}
Suppose now that $\delta_1 > 0$.
Partition $A$ into three subsets $A_1$, $A_2$ and $A_3$ in the following way: $A_1:=A\cap f_n^{-1}F_n$, $A_2:=A\cap f_n^{-1}F_n \phi_n(h_0)$ and $A_3:=A\setminus (A_1\sqcup A_2)$.
In other words, $f_n A_1\subset F_n$, $f_n A_2\subset F_n \phi_n(h_0)$, $f_n A_3\cap(F_n \sqcup F_n \phi_n(h_0)) = \varnothing$.

Note that
\begin{equation}
\mu([A_3C_{n+1}^2]_{n+1}) \leqslant \mu([A_3]_n) \leqslant \frac{2n+1}{2r_n+1} \to 0.
\end{equation}
For $A_1$ and $A_2$ we argue as in the proof of Lemma~\ref{LemWM}.
Set $F^\circ_n := \{f\in F_n\mid fS_nS_n^{-1}\subset F_n\}$,
$A_1^\circ:= A_1\cap F^\circ_n$, $B^\circ:=B\cap F^\circ_n$.
We have
\begin{align*}
\mu(T_{g_n}[A_1 C_{n+1}^2]&_{n+1} \cap [B]_n)
= \sum_{h\in H'_n} \mu([\phi_n(h_n) f_n A_1^\circ c_{n+1}(h)]_{n+1} \cap [B^\circ]_n) + \overline{o}(1) \\
&= \sum_{h\in H'_n} \mu([(f_n A_1^\circ s_n(h) s_n(h_n+h)^{-1} \cap B^\circ) c_{n+1}(h)]_{n+1}) + \overline{o}(1) \\
&= \frac{1}{|H_n|} \sum_{h\in H'_n} \mu([f_n A_1^\circ s_n(h) s_n(h_n+h)^{-1} \cap B^\circ]_n) + \overline{o}(1) \\
&= \frac{\delta_1}{|H'_n|} \sum_{h\in H'_n} \lambda_{F_n}(f_n A_1^\circ s_n(h) \cap B^\circ s_n(h_n+h)_n) + \overline{o}(1) \\
&= \delta_1 \int_{S_n\times S_n} f_{A_1 f_n,B} d\nu_n + \overline{o}(1),
\end{align*}
where $\nu_n := \dist_{h\in H'_n}(s_n(h),s_n(h_n+h))$ and $f_{A_1 f_n,B}(x,y) = \lambda_{F_n}(A_1 f_nx\cap By)$.
Write $h_n=(t_n,0)$.
Since $\frac{2r_n-t_n+1}{2r_n+1} = \frac{|H'_n|}{|H_n|}\to \delta_1>0$ and
$$
\nu_n = \frac{1}{2r_n-1} \sum_{i=-r_n}^{r_n} \dist_{-r_n \leqslant t \leqslant r_n-t_n}(s_n(t,i),s_n(t+t_n,i)),
$$
it follows from (\ref{dJlem}) and (\ref{discr_meas}) that
$$
\mu(T_{g_n}[A_1 C_{n+1}^2]_{n+1} \cap [B]_n) = \frac{\delta_1}{\lambda(S_n)^2} \int_{S_n\times S_n} f_{A_1 f_n,B} d\lambda d\lambda + \overline{o}(1).
$$
Now take $A:=A^* C^*_n$ and $B:=B^* C_n$ for some $\xi_{n-1}$-measurable subsets $A^*,B^*\subset F_{n-1}$.
Let $C'_n := C_n\cap F'_n$.
It follows that $\frac{|C'_n|}{|C_n|}\to\delta_2$ and $\mu([A_1]_n\bigtriangleup [A^* C_n]_n)=\overline{o}(1)$.
Hence $\mu([A_1]_n) = \delta_2 \mu([A^*]_{n-1}) + \overline{o}(1)$.
Arguing as in the proof of Lemma~\ref{LemWM} we obtain that
$$
\mu(T_{g_n}[A^*C'_nC^2_{n+1}]_{n+1}\cap [B^*]_{n-1}) = \delta_2 \mu([A^*]_{n-1}) \mu([B^*]_{n-1}) + \overline{o}(1).
$$
Therefore
\begin{equation}
\left| \mu(T_{g_n}[A_1C^2_{n+1}]_{n+1} \cap [B]_n) - \mu([A_1C^2_{n+1}]_{n+1}) \mu([B]_n) \right| \to 0.
\end{equation}
Since $T_{g_n}[A_2]_n = T_{\phi_n(h_n+h_0)}[\phi_n(h_0)^{-1} f_n A_2]$ with $\phi_n(h_0)^{-1} f_n A_2\subset F_n$,  a similar reasoning yields
\begin{equation} \label{last}
\left| \mu(T_{g_n}[A_2C^2_{n+1}]_{n+1} \cap [B]_n) - \mu([A_2C^2_{n+1}]_{n+1}) \mu([B]_n) \right| \to 0.
\end{equation}
Since
$$
[A^*]_{n-1} = [A^* C_n C^1_{n+1}]_{n+1} \sqcup \bigsqcup_{i=1}^3 [A_i C^2_{n+1}]_{n+1} \sqcup [A^* C_n C^3_{n+1}]_{n+1},
$$
it follows from (\ref{first})--(\ref{last}) that
$$
\lim_{n\to\infty} \sup_{A^*,B^*\in\sigma(\xi_{n-1})} \left| \mu(T_{g_n}[A^*]_{n-1}\cap[B^*]_{n-1}) - \mu([A^*]_{n-1}) \mu([B^*]_{n-1}) \right| = 0.
$$
Since $\xi_n$-measurable cylinders generate the entire $\sigma$-algebra $\fr B$ as $n\to\infty$, it follows that $(g_n)_{n=1}^\infty$ is a mixing sequence for $T$, as desired.
\end{proof}

\begin{prop} \label{simpl}
The transformation $T_{(1,0,0)}$ is 2-fold simple and $C(T_{(1,0,0)}) = \{T_g\mid g\in G\}$.
\end{prop}

\begin{proof}
Take an ergodic joining $\nu\in J_2^e(T_{(1,0,0)})$.
Let $K_n := \left[-\frac{a_n}{n^2},\frac{a_n}{n^2}\right]_\Z$, $J_n := \left[-\frac{r_n}{n^2},\frac{r_n}{n^2}\right]_\Z$ and $\Phi_n := K_n + 2\widetilde{a}_n J_n$.
We claim that $\nu$-a.e. point $(x,y)\in X\times X$ is \emph{generic} for $T_{(1,0,0)}\times T_{(1,0,0)}$, i.e. for all cylinders $A,B\subset \bigcup_{n=1}^\infty \sigma(\xi_n)$ we have
\begin{equation} \label{generic}
\nu(A\times B) = \lim_{n\to\infty} \frac{1}{|\Phi_n|} \sum_{i\in\Phi_n} \chi_A (T_{(i,0,0)}x) \chi_B (T_{(i,0,0)}y).
\end{equation}
To see this, we first note that $(\Phi_n)_{n=1}^\infty$ is a F{\o}lner sequence in $\Z$.
Since
$$
\frac{a_n}{n^2} + \frac{2\widetilde{a}_n r_n}{n^2} < \frac{\widetilde{a}_n(2r_n+1)}{n^2} < \frac{2a_{n+1}}{(n+1)^2},
$$
it follows that $\Phi_n\subset K_{n+1} + K_{n+1}$ and hence $\bigcup_{m=1}^n \Phi_m\subset K_{n+1} + K_{n+1}$.
This implies that $\left| \Phi_{n+1} + \bigcup_{m\leqslant n} \Phi_m \right| \leqslant 3\left|\Phi_{n+1}\right|$ for every $n\in\N$, i.e. Shulman's condition \cite{Li} is satisfied for $(\Phi_n)_{n=1}^\infty$.
By \cite{Li}, the pointwise ergodic theorem holds along $(\Phi_n)_{n=1}^\infty$ for any ergodic transformation.
Since $T\times T$ is $\nu$-ergodic, (\ref{generic}) holds for $\nu$-a.a.\ $(x,y)\in X\times X$ and for every pair of cylinders $A,B\subset X$ from $\bigcup_{n=1}^\infty \sigma(\xi_n)$.

Fix a generic point $(x,y)\in X\times X$.
Since $x,y\in X_n$ for all sufficiently large $n$ and we have the following expansion
\begin{align*}
x &= (f_n, c_{n+1}(h_n), c_{n+2}(h_{n+1}),\ldots),\\
y &= (f'_n, c_{n+1}(h'_n), c_{n+2}(h'_{n+1}),\ldots)
\end{align*}
with $f_n,f'_n\in F_n$, $h_i,h'_i\in H_i$, $i\geqslant n$.
We let $H_n^- := \left[ -(1-\frac{1}{n^2})r_n, (1-\frac{1}{n^2})r_n \right]_\Z^2 \subset H_n$.
Since the marginals of $\nu$ both equal to $\mu$, we may assume without loss of generality that $h_n,h'_n\in H_n^-$.
Indeed, $$\mu(\{x=(f_n, c_{n+1}(h_n), c_{n+2}(h_{n+1}),\ldots)\in X_n \mid h_n\not\in H_i^-\})<\frac{2}{i^2},$$ and hence by Borel-Cantelli lemma for $\mu$-a.e.\ $x\in X_n$ and all but finitely many $i$ we have $h_i\in H_i^-$.
Then we may replace $x=(f_n, c_{n+1}(h_n), c_{n+2}(h_{n+1}),\ldots)\in X_n$ with $x=(f_n c_{n+1}(h_n)\cdots c_m(h_{m-1}), c_{m+1}(h_m),\ldots)\in X_m$ for some $m>n$ if necessary.
Similarly, $h'_n\in H_n^-$.

This implies, in turn, that
\begin{multline} \label{H_n_incl}
f_{n+1} = f_n c_{n+1}(h_n) \in
\widetilde{F}_n \phi_n(H_n^-) \subset [ -c_n, c_n ]_\Z \times [ -c_n, c_n ]_\R \times \Z_2 ,
\end{multline}
where $c_n = \widetilde{a}_n(1+2r_n(1-\frac{1}{n^2}))$,
and, similarly, $f'_{n+1} \in [ -c_n, c_n ]_\Z \times [ -c_n, c_n ]_\R \times \Z_2 $.

Given $g\in\Phi_n$, there are some uniquely determined $k\in K_n$ and $j\in J_n$ such that $g = k + 2\widetilde{a}_n j$, i.e. $(g,0,0) = (k,0,0) \phi_n(j,0)$.
Moreover, $(j,0) + h_n \in H_n$ since $h_n\in H_n^-$.
It also follows from (\ref{H_n_incl}) that
\begin{equation} \label{b_incl}
(k,0,0)f_n S_n S_n^{\pm 1} \subset F_n.
\end{equation}

Take $g\in\Phi_n$ and calculate $T_{(g,0,0)}x$.
$$
x = (f_n, c_{n+1}(h_n),\ldots) = (f_n c_{n+1}(h_n),\ldots) = (f_n s_n(h_n) \phi_n(h_n),\ldots).
$$
\begin{align*}
(g,0,0)f_n s_n(h_n) \phi_n(h_n)
&= (k,0,0) \phi_n(j,0) f_n s_n(h_n) \phi_n(h_n)\\
&= (k,0,0) f_n s_n(h_n) \phi_n((j,0)+h_n)\\
&= (k,0,0) f_n s_n(h_n) s_n((j,0)+h_n)^{-1} c_{n+1}((j,0)+h_n)\\
&= d c_{n+1}((j,0)+h_n),
\end{align*}
where $d := (k,0,0) f_n s_n(h_n) s_n((j,0)+h_n)^{-1} \in F_n$ by (\ref{b_incl}).
This means that $T_{(g,0,0)}x = (d,\ldots)\in X_n$.
Similarly,
$$
(g,0,0)f'_n, s_n(h'_n) \phi_n(h'_n) = d' c_{n+1}((j,0)+h'_n)
$$
with $d' := (b,0,0) f'_n s_n(h'_n) s_n((t,0)+h'_n)^{-1} \in F_n$.

Now take any $\xi_{n-2}$-measurable subsets $A^*,B^*\subset F_{n-2}$ and set $A:=A^*C_{n-1}C_n$, $B:=B^*C_{n-1}C_n$.
\begin{align*}
\nu([A^*]_{n-2}&\times[B^*]_{n-2}) = \nu([A]_{n}\times[B]_{n}) \\
&= \lim_{n\to\infty} \frac{\left|\{ g\in\Phi_n \mid T_{(g,0,0)}x\in[A]_n, T_{(g,0,0)}y\in[B]_n \}\right|}{|\Phi_n|} \\
&= \lim_{n\to\infty} \frac{\left|\{ g\in\Phi_n \mid d\in A, d'\in B \}\right|}{|\Phi_n|} \\
&= \lim_{n\to\infty} \frac{1}{|K_n|} \sum_{k\in K_n}  \frac{\left|\{ j\in J_n \mid d\in A, d'\in B \}\right|}{|J_n|} \\
&= \lim_{n\to\infty} \frac{1}{|K_n|} \sum_{k\in K_n} \zeta_n (A^{-1}(k,0,0)f_n s_n(h_n) \times B^{-1}(k,0,0)f'_n s_n(h'_n)),
\end{align*}
where $\zeta_n := \dist_{j\in J_n} (s_n((j,0)+h_n),s_n((j,0)+h'_n))$.
We consider separately two cases.

\textit{First case}.
Suppose first that $h_n\neq h'_n$ for infinitely many, say \emph{bad} $n$.
Since $|J_n|\geqslant \frac{r_n}{n^2}$ it follows from (\ref{dJlem}) that $\|\zeta_n - \kappa_{D_n}\times \kappa_{D_n}\|<\frac{1}{n}$.
Moreover, it follows from (\ref{b_incl}) and (\ref{discr_meas}) (we need (\ref{discr_meas}) for $A^{-1}(k,0,0)f_ns_n(h)$) that $$\kappa_{D_n}(A^{-1}(k,0,0)f_ns_n(h)) = \lambda_{S_n}(A^{-1}(k,0,0)f_ns_n(h)) + \overline{o}(1).$$
Hence
\begin{align*}
&\frac{1}{|K_n|} \sum_{k\in K_n} \zeta_n (A^{-1}(k,0,0)f_n s_n(h_n) \times B^{-1}(k,0,0)f'_n s_n(h'_n))\\
&= \frac{1}{|K_n|} \sum_{k\in K_n} \kappa_{D_n}(A^{-1}(k,0,0)f_n s_n(h_n)) \kappa_{D_n}(B^{-1}(k,0,0)f'_n s_n(h'_n)) + \overline{o}(1) \\
&= \frac{1}{|K_n|} \sum_{k\in K_n} \lambda_{S_n}(A^{-1}(k,0,0)f_n s_n(h_n)) \lambda_{S_n}(B^{-1}(k,0,0)f'_n s_n(h'_n)) + \overline{o}(1)
\end{align*}
Now we derive from Lemma~\ref{mainLem}(ii) that
\begin{multline*}
\lambda_{S_n}\left(A^{-1}(k,0,0)f_n s_n(h_n)\right)
= \frac{\lambda\left(A^{-1}(k,0,0)f_n s_n(h_n) \cap S_n\right)}{\lambda(S_n)} \\
= \frac{\lambda\left(A \cap (k,0,0)f_n s_n(h_n) S_n\right)}{\lambda(S_n)}
= \lambda_{F_{n-2}}(A^*) + \overline{o}(1)
\end{multline*}
and, in a similar way, $\lambda_{S_n}\left(B^{-1}(b,0,0)f'_n s_n(h'_n)\right) = \lambda_{F_{n-2}}(B^*) + \overline{o}(1)$.
Hence
$$
\nu([A^*]_{n-2}\times [B^*]_{n-2}) = \lambda_{F_{n-2}}(A^*)\lambda_{F_{n-2}}(B^*) + \overline{o}(1) = \mu([A^*]_{n-2})\mu([B^*]_{n-2}) + \overline{o}(1)
$$
for all bad $n$ and all $\xi_{n-2}$-measurable subsets $A^*,B^*\subset F_{n-2}$.
Since any measurable set can be approximated by $[A^*]_{n-2}$, it follows that in this case $\nu = \mu\times\mu$.

\textit{Second case}.
Now we consider the case where $h_n = h'_n$ for all $n$ greater than some $N$.
Then it is easy to see that $y = T_k x$, where $k=f'_N f^{-1}_N\in G$ and then it follows immediately that $(x,y)$ is generic for the off-diagonal joining $\mu_{T_k}$:
\begin{multline*}
\nu([A]_{n}\times[B]_{n})
= \lim_{n\to\infty} \frac{1}{|\Phi_n|} \sum_{i\in\Phi_n} \chi_{[A]_n} (T_{(i,0,0)}x) \chi_{[B]_n} (T_{(i,0,0)}T_k x) = \\
= \lim_{n\to\infty} \frac{1}{|\Phi_n|} \sum_{i\in\Phi_n} \chi_{[A]_n \cap T_k^{-1}[B]_n} (T_{(i,0,0)} x)
= \mu([A]_n \cap T_k^{-1}[B]_n) = \mu_{T_k}([A]_{n}\times[B]_{n})
\end{multline*}
for all $A,B\in\sigma(\xi_n)$, since $\nu$ projects onto $\mu$.
Since each measurable set can be approximated by cylinder sets, we deduce that in this case $\nu = \mu_{T_k}$ with $k\in G$.
\end{proof}

\begin{proof}[Proof of Theorem~\ref{MainTh}]
follows now from Veech's theorem, Propositions~\ref{mixingProp}, \ref{simpl} and the fact that $\fr F_{G_a}$ and $\fr F_{G_b}$ are isomorphic if and only if $G_a$ and $G_b$ are conjugate in $G$ \cite[Corollary~3.3]{dJR}.
It is clear, that $G_b = h G_a h^{-1}$ with $h=(0,\frac{a+b}{2},1)$.
\end{proof}

Notice that with some additional conditions on $s_n$ in Lemma~\ref{dJlemma} (cf. \cite[Lemma~2.3]{Da_fs}) one
can show that $T_{(1,0,0)}$ is actually mixing of all orders.

\section{Concluding remarks}

If we replace $G=\Z\times(\R\rtimes\Z_2)$ with $\Gamma:=\R\times(\R\rtimes\Z_2)$ and apply the same construction (with obvious minor changes) we obtain a probability preserving $\Gamma$-action $R$ such that the flow $(R_{(t,0,0)})_{t\in\R}$ is 2-fold simple mixing and its centralizer coincides with the entire $\Gamma$-action.
This gives an example of 2-fold simple mixing \emph{flow} with uncountably many prime factors.
By \cite{Ry}, each 2-fold simple flow is simple.
For the definitions of higher order simplicity we refer to \cite{dJR}.
Moreover, since $\Z\subset\R$ is a closed cocompact subgroup, the corresponding $\Z$-subaction is also 2-fold simple and $C(R_{(1,0,0)}) = \{R_g\mid g\in \Gamma\}$ by \cite[Theorem~6.1]{dJR}.
Thus we get examples of two nonisomorphic 2-fold simple transformations with uncountably many prime factors: $R_{(1,0,0)}$ is embeddable into a flow while $T_{(1,0,0)}$ is not.

\end{document}